\documentclass[11pt]{amsart}

\usepackage{xspace}
\usepackage{amssymb}
\usepackage[frame,ps,matrix,arrow,curve,rotate]{xy}
\usepackage{epsfig}
\usepackage{graphicx}

\usepackage{setspace}
\newtheorem{theorem}{Theorem}[section]
\newtheorem{proposition}[theorem]{Proposition}
\newtheorem{conjecture}[theorem]{Conjecture}
\newtheorem{question}[theorem]{Question}
\newtheorem{lemma}[theorem]{Lemma}
\newtheorem{corollary}[theorem]{Corollary}

\theoremstyle{remark}

\theoremstyle{definition}

\def\P{{\bf P}}
\def\C{{\bf C}}

\def\Q{{\bf{Q}}}
\def\w{\bigwedge\nolimits}

\def\O{{\mathcal O}}

\def\pG{\mathbb G}
\def\uU{\mathcal U}
\def\lL{\mathcal L}
\def\mM{\mathcal M}

\begin{document}

\title{Positivity results for spaces of rational curves}

\author{Roya Beheshti, Eric Riedl}
\date{}
\address{Department of Mathematics, Washington University in St. Louis}
\email{beheshti@wustl.edu}
\address{Department of Mathematics, the University of Notre Dame}
\email{eriedl@nd.edu}

\begin{abstract}
Let $X$ be a very general hypersurface of degree $d$ in $\P^n$. We investigate positivity properties of the spaces $R_e(X)$ of degree $e$ rational curves in $X$. We show that for small $e$, $R_e(X)$ has no rational curves meeting the locus of smooth embedded curves. We show that for $n \leq d$, there are no rational curves other than lines in the locus $Y \subset X$ swept out by lines. And we exhibit differential forms on a smooth compactification of $R_e(X)$ for every $e$ and $n-2 \geq d \geq \frac{n+1}{2}$.
\end{abstract}

\maketitle

\section{Introduction}
We work over $\C$ the field of complex numbers. Let $X$ be a smooth hypersurface of degree $d$ in $\P^n$, and for $e\geq 1$, denote by $R_e(X)$ the space of smooth rational curves of degree $e$ on $X$. In this paper, we study some geometric properties of $R_e(X)$. 

\begin{question}\label{question}
For which $d$, $n$ and $e$ with $n \geq d$ does the very general hypersurface $X \subset \P^n$ of degree $d$ have a rational curve in $R_e(X)$?
\end{question}

One major motivation for considering Question \ref{question} is to study rational surfaces in Fano hypersurfaces. A rational curve in $R_e(X)$ gives a rational surface in $X$, and conversely a rational surface in $X$ gives a non-constant map from $\P^1$ to a compactification of $R_e(X)$ for some $e\geq 1$. It is known that if $d < < \sqrt{n}$ every smooth hypersurface of degree $d$ contains rational surfaces, but it is not known if the same holds for higher degree Fano hypersurfaces.  It is conjectured that when $d=n \geq 5$, $X$ is not covered by rational surfaces: 
\begin{conjecture}\label{biratConj}
A very general hypersurface of degree $n$ in $\P^n$ is not covered by rational surfaces if $n \geq 5$.
\end{conjecture}

Conjecture \ref{biratConj} has important implications in understanding birational properties of varieties. Recall that a variety is unirational if it is rationally dominated by projective space, and it is rationally connected if there is a rational curve through two general points. It is uniruled if there is a rational curve through a general point. Every unirational variety is rationally connected, and indeed, is swept out by rational surfaces. It is expected that there exist rationally connected varieties which are not unirational, but to date, no examples of this have been proven. Proving Conjecture \ref{biratConj} would prove that a very general hypersurface of degree $n$ in $\P^n$ is not unirational. Since every hypersurface of degree $d \leq n$ is rationally connected, this  would give an example of a variety which is rationally connected but not unirational.

Riedl and Yang \cite{RY2} answer Question \ref{question} for $e=1$ by showing that if $n \leq \frac{d^2+3d+6}{6}$ there is no rational curve in $R_1(X)$, the space of lines on $X$. Beheshti \cite{beheshti} proves that $R_e(X)$ is not uniruled for $\frac{n+1}{2} \leq d \leq n-3$. Beheshti and Starr \cite{beheshtiStarr} consider the special case $d=n$ and prove that $X$ is not swept out by del Pezzos or rational surfaces ruled by curves of degree up to $n$. However, Question \ref{question} remains open and is presumably quite difficult in general. 


In this paper, we prove a few results about the positivity of $R_e(X)$. 
In Section 2, we consider the Kontsevich moduli space of stable maps $\overline{\mathcal M}_e(X)$, a compactification of 
$R_e(X)$, and show the following: 

\begin{theorem}\label{intro-main}
Fix an integer $n \geq 20$ and $e$, and suppose $d \leq n $ satisfies 
$$d^2+(2e-1)d \geq e(e+1)n - 3e(e-1) +2 \ \ \ \mbox{ if $e \geq 3$}$$
$$ d^2+(2e+1)d \geq (e+1)(e+2)n -3e(e+1)+2 \ \ \ \mbox{ if $e = 1,2$. } $$
Then for a very general hypersurface $X$ of degree $d$ in $\P^n$, there is no non-constant morphism  $\P^1 \to \overline{\mathcal M}_e(X)$ whose image intersects the locus of embedded rational curves in $X$.
\end{theorem}

This is a substantial generalization of results in Riedl-Yang, since it applies to curves with $e > 1$. Unlike results in \cite{beheshtiStarr}, it applies to hypersurfaces with $n > d$, and proves that there are no rational curves meeting the locus of smooth curves in $R_e(X)$, instead of merely proving that locus of $R_e(X)$ covered by rational curves does not sweep out $X$. It also generalizes results of \cite{beheshti}. 

In \cite[Corollary 10.10]{starr}, Starr computes the $\Q$-divisor class of the first chern class of the dualizing sheaf of $\overline{\mathcal M}_e(X)$ for a general 
hypersurface $X$, and shows that the canonical divisor of $\overline{\mathcal M}_e(X)$  is big for $n$ less that about $d^2$. 
If $\overline{\mM}_e(X)$ were irreducible, of the expected dimension, and had canonical singularities, then $\overline{\mM}_e(X)$ would be of general type for $d$ large, $n \geq d+6$, and $2n + 2 \leq d^2 + d$. Results of Harris, Roth, Starr, Beheshti, Kumar, Riedl, and Yang \cite{HRS, beheshtiKumar, riedl-yang1} prove that $\overline{\mM}_e(X)$ is irreducible and of the expected dimension for $d \leq n-2$. However, $\overline{\mM}_e(X)$ is not always known to have canonical singularities. The best current results are for $e+d \leq n$, due to work of Starr \cite{starr}.

Theorem \ref{intro-main}, along with results of \cite{starr} and \cite{beheshti}, suggests that there should be non-zero 
pluri-canonical forms on $\overline{\mathcal M}_e(X)$ if $d \geq (n+1)/2$ or if $d<(n+1)/2$ and the inequality of 
Theorem \ref{intro-main} is satisfied. In \cite{dejong-starr}, de Jong and Starr give a general construction of differential forms on any desingularization of $\overline{M}_e(X)$, the coarse moduli scheme of $\overline{\mM}_e(X)$, and use this to construct non-zero pluri-canonical forms on these schemes when $X$ is a general cubic fourfold and $e \geq 5$ is an odd number. In Section 3, we use their construction and show 
\begin{proposition}
If $\frac{n+1}{2} \leq d \leq n-2$ and $X$ is any smooth hypersurface of degree 
$d$ in $\P^n$, then there are non-zero differential forms on any disingularization of $\overline{M}_e(X)$ for every $e \geq 1$.
\end{proposition}

It is interesting to investigate whether these forms can be used to construct non-zero pluri-canonical forms when $d \geq (n+1)/2$. 

Finally, in the last section we generalize results of \cite{RY2} in a different direction, and show 
\begin{theorem}\label{gonal}
If $n \leq \frac{d(d+3)}{6}+1-\frac{k}{3}$, then a very general degree $d$ hypersurface $X$ in $\P^n$ contains no $k$-gonal curves in $F_1(X)$.
\end{theorem}

If $n \leq d$, the lines on $X$ will sweep out a proper subvariety $Y \subset X$. Results of Clemens and Ran \cite{clemensRan} seem to suggest that $Y$ will be the ``most negative'' subvariety of $X$. They prove that for $d \geq \frac{3n+1}{2}$, any subvarieties of $X$ without effective canonical bundle (such as rational curves), must lie in $Y$. The following corollary of Theorem \ref{gonal} proves that in contrast to this, $Y$ never contains rational curves for $n \leq d$. 

\begin{corollary}
\label{thm-kgonalCurvesInY}
Let $X \subset \P^n$ be a very general degree $d$ hypersurface with $d \geq n$. Let $Y \subset X$ be the locus swept out by lines. Then $Y$ contains no $k$-gonal curves other than lines if $k \leq \frac{d^2+3d+6-6n}{2(n-1)}$. In particular, $Y$ contains no rational curves other than lines.
\end{corollary}

\textbf{Acknowledgements:} We are greatful to the anonymous referee for numerous helpful suggestions and corrections. The second author was partially supported by the NSF RTG grant number DMS-1246844 while doing much of this research.

\section{Rational curves in $R_e(X)$}\label{ReSection}
For a hypersurface $X \subset \P^n$, we denote by $\overline{\mathcal M}_e(X)$ the Kontsevich moduli space of stable maps of degree $e$ from curves of genus 0 to $X$.  The goal of this section is to prove the following.

\begin{theorem}\label{main}
Fix an integer $e$ and $n$ ($n \geq 20$), and suppose $d\leq n$ satisfies the following:  
$$d^2+(2e-1)d \geq e(e+1)n - 3e(e-1) +2 \ \ \ \mbox{ if $e \geq 3$}$$
$$ d^2+(2e+1)d \geq (e+1)(e+2)n -3e(e+1)+2 \ \ \ \mbox{ if $e = 1,2$. } $$
Then for a very general hypersurface $X$ of degree $d$ in $\P^n$, there is no non-constant morphism  $\P^1 \to \overline{\mathcal M}_e(X)$ whose image intersects the locus of embedded smooth rational curves in $X$.
\end{theorem}

Before proving Theorem \ref{main}, we set some notation. Let $S$ be a smooth rational surface admitting a map to $\P^1$ with general fiber $C$ isomorphic to $\P^1$. Let $f: S \to \P^n$ be a generically finite morphism whose image is contained in a smooth hypersurface $X$ of degree $d$ in $\P^n$. We denote by 
$N_{f, \P^n}$ the normal sheaf of $f$, i.e. the cokernel of the map $T_S \to f^* T_{\P^n}$. Similary, we denote by $N_{f,X}$ the normal sheaf of $f$ considered as 
a morphism from $S$ to $X$. There is a short exact sequence 
\begin{equation}\label{normal}
0 \to N_{f, X} \to N_{f, \P^n} \to f^*\O_{\P^n}(d) \to 0
\end{equation}

The strategy of the proof will be to compute the Euler characteristic of $N_{f,X}(t) \otimes I_{C/S}$ for various values of $t$. The main technique will be to argue that for $X$ a general hypersurface of the appropriate degree and for $t$ carefully chosen, the Euler characteristic will be positive. We contrast this with the following direct computation of the Euler characteristic.

\begin{proposition}
\label{prop-Nidealtwist}
Assume $B$ is a smooth projective curve and $S$ a smooth surface admitting a fibration $\pi: S \to B$ with rational fibers. Let $f: S \to X$ be a map from $S$ to a smooth degree $d$ hypersurface $X \subset \P^n$, and let $C$ be a general fiber of $\pi$. Suppose $f|_C$ is an embedding. Let $H = f^* \O(1)$ and $K$ be the canonical class of $S$. Then we have
\[ \chi(N_{f,X}(t) \otimes I_{C/S}) = \frac{b_t}{2} H^2 + \frac{c_t}{2} H \cdot K - 2K^2 + d_t ,\]
where
\[ b_t = (n-3)t^2+2(n+1-d)t + n+1-d^2 , \]
\[ c_t = -(n-5)t+d-n-1 \]
and
\[ d_t = (-(n-3)t-n-1+d) \; H \cdot C + (n+9) \; \chi(\O_S) -n+5 . \]
\end{proposition}
\begin{proof}
This is a calculation involving additivity of the Euler characteristic. Using the exact sequence (\ref{normal}), we see
\[ \chi(N_{f,X} \otimes I_{C/S}(t)) = \chi(N_{f,\P^n} \otimes I_{C/S}(t)) - \chi(f^* \O_{\P^n}(d+t) \otimes I_{C/S}) . \]
Using the sequence defining $N_{f,\P^n}$ and the Euler sequence for $\P^n$, we see that 
\begin{equation*}
\begin{split}
 &\chi(N_{f,X} \otimes I_{C/S}(t)) \\
 &= \chi(T_{\P^n} \otimes I_{C/S}(t)) - \chi(T_S \otimes I_{C/S}(t)) - \chi(f^* \O_{\P^n}(d+t) \otimes I_{C/S})\\
  &=  (n+1) \chi(f^* I_{C/S}(t+1)) - \chi(I_{C/S}(t)) - \chi(f^* \O_{\P^n}(d+t) \otimes I_{C/S})  - \chi(T_S \otimes I_{C/S}(t)).
 \end{split}
 \end{equation*}
 Using Hirzebruch-Riemann-Roch, this becomes 
 \begin{equation}\label{eulerchar}
 \begin{split}
 & (n+1) \frac{((t+1)H-C) \cdot ((t+1)H - C - K)}{2} - \frac{(tH-C) \cdot (tH - C - K)}{2} \\
& - \frac{((t+d)H-C) \cdot ((t+d)H - C - K)}{2} + (n-1)\chi(\O_S) -  \chi(T_S \otimes I_{C/S}(t)). 
\end{split}
\end{equation}
We have $$c_1(T_S \otimes I_{C/S}(t)) = c_1(T_S) +2 (-C+tH) = -K+2(-C+tH)$$ and $$c_2(T_S \otimes I_{C/S}(t)) 
=c_2(T_S) -K \cdot (-C+tH)+(-C+tH)^2.$$ Via the splitting principle, we can introduce variables $\alpha$ and $\beta$ with $\alpha + \beta = c_1(T_S \otimes I_{C/S}(t))$ and $\alpha \beta = c_2(T_S \otimes I_{C/S}(t))$. By Hirzebruch-Riemann-Roch, we get
\[ \chi(T_S \otimes I_{C/S}(t)) = 2\chi(\O_S) + \frac{\alpha (\alpha-K)}{2} + \frac{\beta (\beta-K)}{2} \]
\[ = 2\chi(\O_S) -\alpha \beta + \frac{(\alpha+\beta)^2-(\alpha+\beta) \cdot K}{2} . \] 
Plugging in the chern classes of $T_S \otimes I_{C/S}(t)$ for $\alpha + \beta$ and $\alpha \beta$, we obtain
$$    \chi(T_S \otimes I_{C/S}(t)) = K^2-c_2(T_S)-2tH\cdot K+2K\cdot C +t^2H^2-2tH \cdot C +2 \chi(\O_S).$$
We have $K\cdot C = -2$. So putting the above equation in (\ref{eulerchar}), collecting like terms, and using Noether's formula $c_2(T_S) = 12 \chi(\O_S) - K^2$, we obtain the result.
\end{proof}

\begin{corollary}\label{euler-neg}
Assume $t \geq 1$, $n \geq 20$ and $d \leq n$. If $S$ contains no $-1$-curves contracted by both $f$ and $\pi$, and $$d^2+d(2t+1) \geq n(t+2)(t+1) -3t^2-3t+2,$$  then $\chi(N_{f,X} \otimes I_{C/S}(t)) < 0$.
\end{corollary}
\begin{proof}
This is an intersection theory calculation on $S$. First observe that $2H+2C+K$ is basepoint free, and hence, nef. If it were not, then by Reider's theorem \cite{reider}, there would be an effective divisor $E$ on $S$ with either $E \cdot (2H+2C) = 1$ and $E^2 = 0$ or $E \cdot (2H+2C) = 0$ and $E^2 = -1$. The first case is impossible since $E \cdot (2H+2C)$ must be even. In the second case, we would have $H \cdot E = 0 = C \cdot E$, which implies that $E$ is a $(-1)$-curve contracted by both $\pi$ and $f$, and contradicts our assumption. It follows that $2H+2C+K$ is nef. So
\begin{equation} \label{K2ineq}
\begin{split}
0 \leq (2H+2C+K)^2 & = 4H^2+K^2+4H \cdot K + 8 H \cdot C -8 
\end{split}
\end{equation}
since $K \cdot C = -2$ and $C^2=0$.
Note that $H^1(f^* \O(-1)) = 0$ by Kodaira vanishing, so by Hirzebruch-Riemann-Roch and the fact that $\chi(\O_S) \leq 1$ we see that
\begin{equation} \label{HHplusKineq}
H \cdot (H+K) = 2 \chi(f^* \O_X(-1)) - 2 \chi (\O_S)\geq -2 \chi(\O_S) \geq -2 . 
\end{equation}
By Proposition \ref{prop-Nidealtwist} and the relation (\ref{K2ineq}) we see that
\begin{equation*}
\begin{split} 
\chi(N_{f,X}(t) \otimes I_{C/S}) & = \frac{b_t}{2} H^2 + \frac{c_t}{2} H \cdot K - 2K^2 + d_t \\
& \leq  \frac{b_t}{2} H^2 + \frac{c_t}{2} H \cdot K+8H\cdot (H+K)+16 H \cdot C-16+d_t\\
& = \frac{b_t-c_t}{2} H^2+\frac{c_t+16}{2} (H \cdot (H+K)+2)+d_t-c_t+16H \cdot C-32. \\
\end{split}
\end{equation*}
Note that since by our assumption $d \leq n$ and $n \geq 20$, we have $c_t +16 \leq 0$, so the above inequality and (\ref{HHplusKineq}) give
$$\chi(N_{f,X}(t) \otimes I_{C/S})  \leq \frac{b_t-c_t}{2} H^2+d_t-c_t+16H \cdot C-32. 
$$
We see that $b_t - c_t \leq 0$ precisely when $(n-3)t^2+(3n-3-2d)t+2n+2-d^2-d \leq 0$, or
\[ d^2+d(2t+1) \geq n(t+2)(t+1) -3t^2-3t+2 . \]
It remains to show that $d_t-c_t+16H \cdot C-32 < 0$.   
By Proposition \ref{prop-Nidealtwist}  we have 
\begin{equation*}
\begin{split}
& d_t-c_t+16H \cdot C-32 \\& = -(n-d+t(n-3)-15)(H \cdot C -1)-2t+ (n+9) \chi(\O_S) -(n-5) -16 \\
& \leq  -(n-d+t(n-3)-15)(H \cdot C -1)-2t-2.
\end{split}
\end{equation*}
We see that this will be negative if $n-d+t(n-3)-15 \geq 0$, which will happen for $n \geq 20$ since 
$t \geq 1$ and $d\leq n$.

\end{proof}

\medskip
\begin{proposition}\label{euler-pos}

Fix $e \geq 1$. Let $m=e-1$ if $e \geq 3$ and $m=e$ if $e=1,2$. Assume that for a very general hypersurface $X$ of degree $d$, there is a map $\phi: \P^1 \to \overline{\mathcal M}_{e}(X)$ whose image intersects the locus of embedded smooth rational curves. Then 
there is a smooth surface 
$S$ with two morphisms $f:S \to X$ and $\pi: S \to \P^1$ such that $f$ 
maps a general fiber $C$ of $\pi$ isomorphically to a smooth rational curve of degree $e$ in $X$ and 
\begin{enumerate}
\item[(a)] 
The image 
of the pull-back map $H^0(X, \O_{X}(d)) \to H^0(S, f^*\O_{\P^n}(d))$ is contained in the image of the map 
$$H^0(S, N_{f, \P^n}) \to H^0(S, f^*\O_{\P^n}(d)).$$
\item[(b)] The restriction map $H^0(S,N_{f,X}(m)) \to H^0(C,N_{f,X}(m)|_C)$ is surjective. \\
\item[(c)] The Euler characteristic $\chi(N_{f,X}(m) \otimes I_{C/S}) \geq 0$.
\end{enumerate} 
\end{proposition}
\begin{proof}

Our assumptions imply that there is an irreducible quasi-projective variety $Z$  and a morphism $\phi: Z \times \P^1 \to \overline{\mathcal M}_{e}(\P^n)$ such that the following hold 
\begin{itemize}
\item For each $z_0 \in Z$, the morphism $\phi_{z_0}:=\phi(z_0,b): \P^1 \to \overline{\mathcal M}_{e}(\P^n)$ is a non-constant morphism whose image intersects the locus of embedded smooth rational curves.
\item  For a very general $X$, there is $z\in Z$ such that the image of $\phi_z$ is contained in $\overline{\mathcal M}_{e}(X)$. 
\end{itemize}
Replacing $Z$ with an open subset we may assume $Z$ is non-singular. 

Let $\P^N$ be the projective space parametrizing hypersurfaces of degree $d$ in $\P^n$ and $U \subset \P^N$ the open subset  parametrizing smooth hypersurfaces. Denote by $I \subset Z \times U$ the incidence correspondence parametrizing pairs 
$(z,[X])$ such that the image of $\phi_z: \P^1 \to \overline{\mathcal M}_{e}(\P^n)$ is contained in $\overline{\mathcal M}_{e}(X)$. Denote by $\pi_1$ and $\pi_2$ the projection maps from $I$ to 
$Z$ and $U$ respectively. By our assumption $\pi_2$ is dominant. Replacing $I$ with an irreducible component which maps dominantly to $U$ under $\pi_2$, we may assume $I$ is irreducible. 
Let $V$ be the pullback of the universal curve $\mathcal C \to \overline{\mathcal M}_e(\P^n)$ to 
$Z \times \P^1$. Denote by $q: V \to \P^n$ the pullback of the universal map $\mathcal C \to \P^n$ to 
$V$ and by $p: V \to Z$ the projection to the first factor.  
Let $d: \tilde{V} \to V$ be a desingularization, and set $\tilde{p} = p \circ d$ and $\tilde{q} = q \circ d$. 
$$
\xymatrix{
\tilde{V} \ar[rdd]_{\tilde{p}}  \ar@/^2pc/[rrr]^{\tilde{q}} \ar[r]^d & V \ar[r] \ar[d] & \mathcal C\ar[d] \ar[r] & \P^n \\
& Z \times \P^1 \ar[r] \ar[d] &  \overline{\mathcal M}_e(\P^n) \\
& Z
}
$$

Let $(z,[X]) \in I$ be a general point. Denote the fiber of $\tilde{p}$ over $z$ by $S$. Since 
$z$ is general in $Z$, by generic smoothness, $S$ is a smooth surface and if $f$ is the restriction of $\tilde{q}$ to $S$, then  $f: S \to X$ maps a general fiber of $S \to \P^1$ isomorphically onto a curve in $X$. 

Denote by $N_{(\tilde{p}, \tilde{q})}$ the normal sheaf of the map $(\tilde{p}, \tilde{q}): \tilde{V} \to Z \times \P^n$. We get a sequence of maps 
$$\rho: \;\;\;\;\; T_{Z,z} \to H^0(S, \tilde{p}^*T_{Z}|_S) \to H^0(S, (\tilde{p},\tilde{q})^*T_{Z \times \P^n}|_S) \to H^0(S, N_{(\tilde{p},\tilde{q})}|_S).$$
Note that $(\tilde{p},\tilde{q})$ is generically finite and $z$ is general, therefore $T_{\tilde{V}}|_S \to  (\tilde{p},\tilde{q})^*T_{Z\times \P^n}|_S$ is injective. So $N_{(\tilde{p},\tilde{q})}|_S$ is isomorphic to $N_{f,\P^n}$:
$$\xymatrix{
& 0 \ar[d] & 0 \ar[d] \\
0 \ar[r] & T_S \ar[r]  \ar[d] & f^*T_{\P^n} \ar[r] \ar[d] & N_{f,\P^n} \ar[r] \ar[d] & 0\\
0 \ar[r] & T_{\tilde{V}}|_S \ar[r] \ar[d] & (\tilde{p},\tilde{q})^*T_{Z\times \P^n}|_S \ar[r]  \ar[d] & N_{(\tilde{p},\tilde{q})}|_S \ar[r] & 0\\
& N_{S/\tilde{V}} \ar[r]^{=}  \ar[d] & \tilde{p}^*T_{Z}|_S \ar[d] \\
 & 0 & 0.
}
$$
There is  a commutative diagram
$$
\xymatrix{
& T_{I, (z,[X])}   \ar[dl]^{d\pi_1} \ar[dr]_{d\pi_2} \\
T_{Z, z}  \ar[d]^{\rho} & & T_{\P^N,[X]}=H^0(X,\O_X(d)) \ar[d]\\
H^0(S,N_{f,\P^n}) \ar[rr] & & H^0(S, f^*\O_X(d)).
}
$$
Since $\pi_2$ is dominant, $d\pi_2$ is surjective, and part (a) follows. 

\bigskip
To prove part (b), we tensor Sequence (\ref{normal}) with $\O_X(m)$ to  get the following short exact sequence 
$$ 0 \to N_{f, X}(m) \to N_{f, \P^n}(m) \to f^*\O_X(d+m) \to 0.$$
Let $C$ be a general fiber of the map $\pi: S \to \P^1$. Restricting the above short exact sequence to $C$, we get a short exact sequence of $\O_C$-modules
$$ 0 \to N_{f, X}(m)|_C \to N_{f, \P^n}(m)|_C \to f^*\O_X(d+m)|_C \to 0.$$
We have the following commutative diagram
{\small{$$
\xymatrix{
H^0(\P^n, \O_{\P^n}(m+1)^{n+1}) \ar[r] \ar[d] &  H^0(\P^n, T_{\P^n}(m)) \ar[d] \ar[r] & 
H^0(S, N_{f, \P^n}(m)) \ar[d] \\
 H^0(C, \O_{\P^n}(m+1)^{n+1}|_C) \ar[r] & H^0(C, T_{\P^n}(m)|_C) \ar[r] &  H^0(C, N_{f, \P^n}(m)|_C).
 }
$$}}
Since $f(C)$ is $(m+1)$-normal, the left vertical map is surjective. Using the Euler sequence and the fact that $H^1(C, \O_C(m))=0$ and $H^1(C, T_S(m)|_C)=0$, we see that the two lower horizontal maps are also surjective. So we have a surjective map  $$H^0(\P^n, \O_{\P^n}(m+1)^{n+1}) \to H^0(C, N_{f, \P^n}(m)|_C).$$
Pick $\alpha \in H^0(C,N_{f,X}(m)|_C)$, and denote the image of $\alpha$ in $H^0(C,N_{f,\P^n}(m)|_C)$ by $\beta$. Let $\tilde{\beta}$ be a lift of $\beta$ to $H^0(\P^n, \O_{\P^n}(m+1)^{n+1})$ under the above map and  $\bar{\beta}$ the image of $\tilde{\beta}$ in $H^0(S, N_{f,\P^n}(m))$. Let $\gamma \in H^0(X, \O_X(m+d))$ be the image of of $\tilde{\beta}$ under the composition
\[ H^0(\P^n, \O_{\P^n}(m+1)^{n+1}) \to H^0(\P^n, T_{\P^n}(m)) \to H^0(X, T_{\P^n}(m)|_X) \to H^0(X, \O_X(m+d)) , \]
where the last map is induced by the sequence
\[ 0 \to T_X \to T_{\P^n}|_X \to \O_X(d) \to 0 .\]
Let $\tilde{\gamma}$ be a preimage of $\gamma$ in $H^0(\P^n, \O_{\P^n}(d+m))$. Then $\bar{\beta}|_C = \beta$ and $f^* \tilde{\gamma}|_C=0$. Since $\tilde{\gamma}|_{f(C)} = 0$, we can view $\tilde{\gamma}$ as an element of $H^0(\P^n,I_{f(C)/\P^n}(d+m))$. Since $I_{f(C)/\P^n}$ is a $m$-regular sheaf, the multiplication map 
$$H^0(\P^n, \O_{\P^n}(d)) \otimes H^0(\P^n, I_{f(C)/\P^n}(m)) \to H^0(\P^n, I_{f(C)/\P^n}(d+m))$$
is surjective, so $\tilde{\gamma}$ can be written as 
$$\tilde{\gamma}= \tilde{\gamma}_1 \tilde{\eta}_1 + \dots +\tilde{\gamma}_k\tilde{\eta}_k$$
where $\tilde{\gamma}_i \in H^0(\P^n,\O_{\P^n}(d))$ and $\tilde{\eta}_i \in H^0(\P^n,I_{f(C)/\P^n}(m))$ for each $i$. So if we view $\tilde{\eta_i}$ as an element of $H^0(\P^n, \O_{\P^n}(m))$ we see that $f^*\tilde{\eta_i}|_C=0$ for each i.  By part (a) the image 
of the map $$f^*: H^0(\P^n, \O_{\P^n}(d)) \to H^0(S, f^*\O_{\P^n}(d))$$ is contained in the image of the map 
$$h: H^0(S, N_{f, \P^n}) \to H^0(S, f^*\O_{\P^n}(d)).$$
So $f^*\tilde{\gamma_i} = h(\bar{\mu}_i)$ for some $\bar{\mu}_i \in H^0(S,N_{f,\P^n})$. Let 
$$\bar{\mu} = \bar{\mu}_1f^*\tilde{\eta}_1+\dots + \bar{\mu}_k f^*\tilde{\eta}_k \in H^0(S,N_{f,\P^n}(m)).$$ 
Then $\bar{\mu}|_C=0$ and since $h(\bar{\mu}-\bar{\beta}) = 0$, $\bar{\mu}-\bar{\beta}$ is the image of a section of $N_{f,X}(m)$ whose restriction to $C$ is $\alpha$.

To prove (c), note that  by the Leray spectral sequence, to show the Euler characteristic is non-negative, it is enough to show that for a general fiber $C$ of $\pi$, (1) $H^1(\P^1,  \pi_*(N_{f,X}(m) \otimes I_{C/S})) =0$ and (2) $R^1\pi_*(N_{f,X}(m)\otimes I_{C/S})=0.$

By part (b), $H^0(S, N_{f,X}(m)) \to 
H^0(C, N_{f,X}(m)|_C)$ is surjective, and $H^1(C, N_{f,X}(m)|_C)=0$, so $$H^1(S, N_{f,X}(m)\otimes I_{C/S}) = H^1(S, N_{f,X}(m)).$$ Thus 
$$H^1(\P^1, \pi_*N_{f,X}(m) \otimes \O_{\P^1}(-1))=H^1(\P^1,  \pi_*N_{f,X}(m)),$$ so $H^1(\P^1,  \pi_*N_{f,X}(m) \otimes \O_{\P^1}(-1)) =0$.  This shows (1).

To show (2), we note that since $X$ is very general, for any morphism $g: \P^1 \to X$, $g^*T_X(1)$ is globally generated by the main result of \cite{clemens}. (See also 
\cite{ein} and \cite{voisin}.) So the restriction of $f^*T_X(m)$ to 
every  irreducible component of every fibre of $\pi$ is globally generated. So $R^1\pi_*T_{X}(m) =0$ and 
$R^1\pi_*N_{f,X}(m) =0$.  
\end{proof}

\begin{proof}[Proof of Theorem \ref{main}] Assume to the contrary that for a very general $X$, there is a non-constant map $\P^1 \to \overline{\mathcal M}_e(X)$ whose image intersects the locus of 
embedded smooth rational curves in $X$.  Let $m=e-1$ if $e\geq 3$ and $m=e$ if $e=1$ or $e=2$. By Proposition \ref{euler-pos}, there is a surface $S$ and map $f: S \to X$ and $\pi:S \to \P^1$ such that $\chi(N_{f,X}(m) \otimes I_{C/S}) \geq 0$. Blowing down, we may assume that there is no $(-1)$-curve  in any fiber of $\pi$ which is contracted by $f$. This contradicts Corollary \ref{euler-neg}.
\end{proof}

A modification of the proof of Proposition \ref{euler-pos} shows that the statement of Theorem \ref{main} remains true for morphisms $\P^1 \to 
\overline{\mathcal M}_e(X)$ whose image intersects the locus of embedded reducible nodal  rational curves in $X$. We sketch the proof here. Fix $e \geq 1$, let $m=e-1$ if $e \geq 3$ and $m=e$ if $e=1$ or 2,  and assume that for a very general hypersurface $X$ of degree $d$, there is a non-consent map $\phi: \P^1 \to \overline{\mathcal M}_e(X)$ whose image intersects the locus of embedded nodal  rational curves in $X$. Then there are positive integers $r$ and $s$ such that $rs \leq e$ and 
for a very general hypersurface $X$ of degree $d$, there is a smooth curve $B$, a degree $r$ morphism $B \to \P^1$, and a morphism $B \to \overline{\mathcal M}_s(X)$ such that 
the image intersects the locus of embedded smooth rational curves of degree $s$ in $X$.

This implies that there are smooth, irreducible, and quasi-projective varieties $P$ and $Z$ ($P$ is just a point in the case of Proposition \ref{euler-pos}), 
a proper morphism $p_1: W \to P$ whose fibers are smooth projective  curves which are degree $r$ covers of $\P^1$ ($r=1$ and $W$ is the projective line in the case of Proposition \ref{euler-pos}), and morphisms $p_2: Z \to P$ and 
$
\phi: Z \times_P W  \to  \overline{\mathcal M}_s(\P^n) 
$ 
with the following property: for every $z \in Z$, 
$$\phi_{z} : p_1^{-1}(p_2(z)) \to \overline{\mathcal M}_s(\P^n)$$
is a morphism which intersects the locus of embedded smooth rational curves, and for a very general  $X$, there is $z$ such that $\phi_{z}$ parametrizes stable maps which are mapped to $X$. We proceed now as in the proof of Proposition \ref{euler-pos} and let $I \subset Z \times U$ be a dominating irreducible component of the the incidence correspondence where $U$ is the locus of smooth hypersurfaces of degree $d$ in $\P^n$. We also let $V$, $S$, and $z$ be as before.

We conclude that if $B=p_1^{-1}(p_2(z))$, then there is a  morphsim of degree $r$ $g: B \to \P^1$, and there are morphisms $f:S \to X$ and $\pi: S \to B$ such that $f$ 
maps a general fiber $C$ of $\pi$ isomorphically onto a smooth rational curve of degree $s$ on $X$, $rs \leq e$ , and 
\begin{enumerate}
\item[(a)] The image 
of the map $H^0(X, \O_{\P^n}(d)) \to H^0(S, f^*\O_{\P^n}(d))$ is contained in the image of the map 
$$H^0(S, N_{f, \P^n}(d)) \to H^0(S, f^*\O_{\P^n}(d)).$$
\item[(b)] The map $H^0(S,N_{f,X}(m)) \to H^0(D,N_{f,X}(m)|_D)$ is surjective where $D$ is a general fiber of $g \circ \pi: S \to \P^1$. \\
\item[(c)] The Euler characteristic $\chi(N_{f,X}(m)\otimes I_{D/S}) \geq 0$.
\end{enumerate} 
Part $(b)$ follows from the proof of Proposition \ref{euler-pos} since the image of $D$ is $e$-regular in $X$. Part $(c)$ follows from a similar argument as in Proposition 
\ref{euler-pos} and the fact that 
if $F$ is a sheaf on $B$ with $H^1(B, F \otimes I_{g^{-1}(p)}) = H^1(B,F)$ for a general $p \in \P^1$, then $H^1(B,F)=0$. Applying Corollary \ref{euler-neg} to $S$ gives the desired result. 

\bigskip
\section{Differential forms on  Kontsevich moduli space}
Let $X$ be a smooth hypersurface of degree $d$ in $\P_{\C}^n$. 
Let $\overline{\mathcal M}_e(X)$ be the Kontsevich moduli space of stable maps of degree $e$ from curves of 
genus zero to $X$, and let $\overline{M}_e(X)$ be the corresponding coarse moduli scheme. There is a universal curve $\pi: \mathcal C \to \overline{\mathcal M}_e(X)$ and an evaluation map 
$ev: \mathcal C \to X$. 


In this section, we use the construction of de Jong and Starr \cite{dejong-starr} to show that 
there are non-zero differential forms on any desingularization of $\overline{M}_e(X)$ when $(n+1)/2 \leq d \leq n-3$. By Corollary 4.3 of \cite{dejong-starr}, for every $i,j \geq 1$, there is a $\C$-linear map 
$$\alpha_{i,j}: H^i(X, \Omega_X^j) \to H^{i-1}(\overline{\mathcal M}_e(X), \Omega_{\overline{\mathcal M}_e(X)}^{j-1}).$$
We consider the map $\alpha_{1,n-2}$, so the above map gives $(n-3)$-forms on the Kontsevich moduli stack. Let $\overline{N}_e(X)$ be a desingularization of $\overline{M}_e(X)$. By \cite[Proposition 3.6]{dejong-starr}, for every $j \geq 0$, there is a linear map 
$$H^0(\overline{\mathcal M}_e(X), \Omega_{\overline{\mathcal M}_e(X)}^{j}) \to  H^{0}(\overline{N}_e(X), \Omega_{\overline{N}_e(X)}^{j}).$$
Composing this with $\alpha_{1,n-2}$, we get a map from $H^1(X, \Omega_X^{n-2})$ to the space of $(n-3)$-forms on $\overline{N}_e(X)$. 

\begin{proposition}\label{forms}
Assume $\frac{n+1}{2} \leq d \leq n-3$ and $X$ is a smooth hypersurface of degree 
$d$ in $\P^n$. If $\overline{N}_e(X)$ is a desingularization of $\overline{M}_e(X)$, then the map $H^1(X, \w^{n-2}\Omega_X) \to H^0(\overline{N}_e(X), \w^{n-3} \Omega_{\overline{N}_e(X)})$ is non-zero for every $e$. 
\end{proposition}

\begin{proof}
Fix $e$ and $X$, and set $\mathcal M = \overline{\mathcal M}_{e}(X)$ and $N= \overline{N}_e(X)$ for simplicity. By \cite[Corollaries 4.2 and 4.3]{dejong-starr} the map $\alpha_{1,n-2}$ factors through the maps 
\begin{equation}\label{form}
H^1(X, \Omega_X^{n-2}) \to H^1(\mathcal C, \Omega_{\mathcal C}^{n-2}) \to H^1(\mathcal C, \pi^*\Omega_{\mathcal M}^{n-3} \otimes \omega_{\pi}) 
\to H^0(\mathcal M, \Omega_{\mathcal M}^{n-3})
\end{equation}
where:

\begin{itemize} 
\item the first map comes from the map $ev^* \Omega_X \to \Omega_{\mathcal C},$
\item the 
second map comes from a map of $\O_{\mathcal C}$-modules 
$$ \Omega_{\mathcal C}^{n-2} \to \pi^*\Omega_{\mathcal M}^{n-3} \otimes \omega_{\pi}$$
which fits into the following short exact sequence over the locus $U$ of embedded smooth curves
$$\pi^*\Omega_{\mathcal M}^{n-2}|_U \to \Omega_{\mathcal C}^{n-2}|_U \to \pi^*\Omega_{\mathcal M}^{n-3} \otimes \omega_{\pi}|_U \to 0,$$
\item and  the last map comes from the Leray spectral sequence and the fact that $R^1\pi_*\omega_{\pi} =\O_{\mathcal M}$.
\end{itemize}
Since $d <n$, there is an irreducible component of $\mathcal M$ whose general point parametrizes an embedded smooth free rational curve of degree $e$ on $X$. Let $C$ be a such a curve.  We denote the stable map corresponding to the isomorphism from $\P^1$ onto $C$ by $[C] \in \mathcal M$, and 
identify the fiber of $\pi: \mathcal C \to \mathcal M$ over $[C]$ with $C$. Since $C$ is free, $\mathcal M$ is smooth at $[C]$ and
$T_{\mathcal M}|_{[C]}= H^0(C,N_{C/X})$. Restricting Sequence (\ref{form}) to $C$, we get the following diagram

$$
\xymatrix{
H^1(X, \Omega_X^{n-2}) \ar[rrr]^{\alpha_{1,n-2}} \ar[d] & & &H^0(\mathcal M, \Omega_{\mathcal M}^{n-3})\ar[d] 
\\
H^1(C, \Omega_X^{n-2}|_C) \ar[r] & H^1(C, \Omega_{\mathcal C}^{n-2}|_C) \ar[r] & H^1(C, \pi^*\Omega_{\mathcal M}^{n-3} \otimes \omega_{\pi}|_C) 
\ar[r]^{\;\;\;\;\;\;\;\;\;\;\;\;\;\;\;\;\; \simeq} &  \Omega_{\mathcal M}^{n-3}|_{[C]} 
}
$$
and we have $\pi^*\Omega_{\mathcal M}^{n-3} \otimes \omega_{\pi}|_C= \w^{n-3} I_{C/\mathcal C} \otimes \Omega_C$. In order to show the statement, it suffices 
to prove that in the above diagram the composition of the maps 
$$H^1(X, \Omega_X^{n-2}) \to  H^1(C, \Omega_X^{n-2}|_C) \to H^1(C, \Omega_{\mathcal C}^{n-2}|_C) $$
$$ \to H^1(C, \pi^*\Omega_{\mathcal M}^{n-3} \otimes \omega_{\pi}|_C) = H^1(C,  \w^{n-3} I_{C/\mathcal C} \otimes \Omega_C)$$
is non-zero. Since $\Omega^{n-3}_{\mathcal M}|_{[C]} = \Omega^{n-3}_{N}|_{[C]}$, this would show the assertion of the theroem. 
From the short exact sequence 
$$0 \to I_{C/X} \otimes \O_C \to \Omega_X|_C \to \Omega_C \to 0,$$ we get the following short exact sequence 
$$0 \to  \w^{n-2} I_{C/X} \otimes \O_C \to \Omega^{n-2}_X|_C \to \w^{n-3} I_{C/X} \otimes \Omega_C \to 0.$$
Similarly, there is an exact sequence 
$$0 \to  \w^{n-2} I_{C/\mathcal C} \otimes \O_C  \to \Omega^{n-2}_{\mathcal C}|_C \to \w^{n-3} I_{C/\mathcal C} \otimes \Omega_C \to 0,$$
and a commutative diagram 
$$\xymatrix{
H^1(C, \Omega^{n-2}_X|_C)  \ar[r] \ar[d] & H^1(C, \w^{n-3} I_{C/X} \otimes \Omega_C) \ar[d] \\
H^1(C, \Omega^{n-2}_{\mathcal C}|_C )  \ar[r] & H^1(C, \w^{n-3} I_{C/\mathcal C} \otimes \Omega_C).
}
$$
So to show the assertion, we show that the composition of the maps 

\begin{equation}\label{h1}
H^1(X, \Omega_X^{n-2}) \to  H^1(C, \Omega_X^{n-2}|_C) \to H^1(C, \w^{n-3} I_{C/X} \otimes \Omega_C) \to  H^1(C,  \w^{n-3} I_{C/\mathcal C} \otimes \Omega_C)
\end{equation}
is non-zero. Note that $$\w^{n-3} I_{C/\mathcal C} \otimes \O_C= \Omega^{n-3}_{\mathcal M}|_{[C]} \otimes \O_C= \w^{n-3} T_{\mathcal M}^\vee |_{[C]} \otimes \O_C= \w^{n-3}H^0(C, N_{C/X})^\vee \otimes \O_C.$$ 
So by Serre duality, the last map in Sequence (\ref{h1}) is the dual of the map 
$$\w^{n-3} H^0(C, N_{C/X}) \to H^0(C, \w^{n-3} N_{C/X}).$$ Since $C$ is free, the above map is surjective, so the last map in Sequence (\ref{h1})
is injective. 
Hence it is enough to  show that under our assumptions, the composition of the maps  
$$H^1(X, \Omega^{n-2}_X) \to H^1(C, \Omega_X^{n-2}|_C) \to H^1(C, \w^{n-3} I_{C/X} \otimes \Omega_C)$$
is non-zero. 

To prove this we consider the short exact sequence 
 $$ 0 \to \O_X(-d) \to \Omega_{\P^n}|_X \to \Omega_X \to 0$$ which gives the following short exact sequence 
\begin{equation}\label{first}
0 \to \Omega^{n-2}_X \to \Omega^{n-1}_{\P^n}|_X \otimes \O_X(d) \to 
\O_X(2d-n-1) \to 0.
\end{equation}
There is also a short exact sequence on $C$
\begin{equation}\label{second}
 0 \to \w^{n-3} I_{C/X} \otimes \Omega_C \to \w^{n-2} I_{C/\P^n} 
\otimes \Omega_C \otimes \O_C(d) \to \O_C(2d-n-1) \to 0,
\end{equation}
and Sequence (\ref{first}) maps to Sequence (\ref{second}). Taking the long exact sequence of cohomology we get a commutative diagram 
$$
\xymatrix{
H^0(X, \O_X(2d-n-1)) \ar[r] \ar[d] & H^1(X, \Omega^{n-2}_X) \ar[d] \\
H^0(C, \O_C(2d-n-1)) \ar[r] & H^1(C, \w^{n-3} I_{C/X} \otimes \Omega_C)
}
$$
and since $2d-n-1 \geq 0$, the left vertical map is non-zero. To show the desired result, we show that the bottom map is injective.  This follows if we 
show $H^0(C,  \w^{n-2} I_{C/\P^n} \otimes \Omega_C \otimes \O_C(d)) =0$. Let 
$$ N_{C/\P^n}= \O_{\P^1}(a_1) \oplus \dots \oplus \O_{\P^1}(a_{n-1}).$$
We have $\w^{n-2} I_{C/\P^n} 
\otimes \Omega_C \otimes \O_C(d)= N_{C/\P^n} \otimes \O_{\P^1}(e(d-n-1))$, so we need to show $a_i < e(n+1-d)$ for each $i$. Since $\sum_i a_i = e(n+1)-2$ and each $a_i$ is at least $e$, we see that $a_i \leq 3e-2$ for all $i$. Thus, $a_i < e(n+1-d)$ provided $d \leq n-2$. This proves the result.

\end{proof}


\section{Gonality and the space swept out by lines}\label{linesSection}

Our goal is to prove the following:
\begin{theorem}
\label{thm-kgonalCurvesInY}
Let $X \subset \P^n$ be a very general degree $d \geq n$ hypersurface. Let $Y \subset X$ be the locus swept out by lines. Then $Y$ contains no $k$-gonal curves other than lines if $k \leq \frac{d^2+3d+6-6n}{2(n-1)}$. In particular, $Y$ contains no rational curves other than lines.
\end{theorem}

We begin with a description of the Fano scheme $F_1(X)$.

\begin{proposition}
Let $X \subset \P^n$ be a general hypersurface. If $2n-d-3 \geq 0$, then the Fano scheme $F_1(X)$ of lines on $X$ is smooth of dimension $2n-d-3$. If $2n-d-3 \geq 0$, then the canonical bundle of $F_1(X)$ is $(\frac{d(d+1)}{2} - n - 1) \sigma_1$, where $\sigma_1$ is the restriction of the divisor on $\pG(1,n)$ of lines meeting a fixed codimension $2$ space.
\end{proposition}

Now we recall the following definitions and results of Bastianelli, De Poi, Ein, Lazarsfeld and Ullery \cite{ELU}. A divisor $D$ is Birationally Very Ample to order $k$ ($BVA_k$) if $D = E + kA$, where $E$ is effective and $A$ is very ample. 

\begin{theorem}[Theorem 1.10 from \cite{ELU}]
If a smooth variety $Z$ satisfies $K_Z$ is $BVA_k$, then $Z$ is not swept out by $k+1$-gonal curves.
\end{theorem}

\begin{corollary}
\label{cor-BVAFano}
If $X$ in $\P^n$ is a general degree $d$ hypersurface with $n \leq \frac{d(d+1)}{2} - k$, then $F_1(X)$ is not swept out by $k$-gonal curves.
\end{corollary}

We need a few basic results about the space of degree $d$ hypersurfaces containing a fixed variety.

\begin{lemma}
\label{lem-pointsConditions}
Any set of $k$ distinct points in $\P^n$ with $k \leq d+1$ imposes $k$ conditions on the space of hypersurfaces of degree $d$.
\end{lemma}
\begin{proof}
First, by a degeneration argument, it suffices to prove the result for $k$ points that lie along a line $\ell$. Then, we can consider the map $\alpha: H^0(\P^n, \O_{\P^n}(d)) \to H^0(\O_{\P^1}(d))$. Since $\ell$ is normal, we see that $\alpha$ is a surjective linear map. Hence, if we fix $k$ points on $\ell$ with $k \leq d+1$, then the space of degree $d$ hypersurfaces in $\P^n$ containing those $k$ points is the preimage under $\alpha$ of all sections of $\O_{\ell}(d)$ that vanish on those $k$ points. The result follows.
\end{proof}

We also need the following standard lemma, described in \cite{riedl-yang1}.
\begin{lemma}
\label{lem-subvarConditions}
Let $Z$ be a variety of dimension $k$. Then it is at least $\binom{d+k}{k}$ conditions for a hypersurface of degree $d$ to contain $Z$.
\end{lemma}

We need the following Proposition from \cite{RY2}.

\begin{proposition}
Let $C \subsetneq \pG(k-1,n)$ be a nonempty variety of $(k-1)$-planes and let $B \subset \pG(k,n)$ be the set of $k$-planes containing the planes in $C$. Then if the codimension of $C$ in $\pG(k-1,n)$ is $\epsilon$, the codimension of $B$ in $\pG(k,n)$ is at most $\epsilon-1$.
\end{proposition}

We use the following corollary which follows immediately from the previous proposition.

\begin{corollary}
\label{cor-cutByHyperplane}
Let $\ell \subset \P^n$ be a line and let $S_k$ be the variety of $k$-planes containing $\ell$. If $C \subset S_{k-1}$ is a nonempty variety of $(k-1)$-planes of codimension $\epsilon > 0$, and $B \subset S_k$ is the set of $k$-planes that contain a plane of $C$, then the codimension of $B$ in $S_k$ is at most $\epsilon -1$.
\end{corollary}

We also discuss the notion of a \emph{parameterized $k$-plane}. A parameterized $k$-plane is a map $\Lambda: \P^k \to \P^n$ defined by linear equations. The set of all parameterized $k$-planes is naturally a $PGL_k$ bundle over the Grassmannian of $k$-planes in $\P^n$. Given a parameterized $k$-plane $\Lambda$, we have a natural map $H^0(\P^n, \O_{\P^n}(d)) \to H^0(\P^k, \O_{\P^k}(d))$ given by pulling polynomials back to $\P^k$ along $\Lambda$. We say that $X' \subset \P^k$ is a parameterized $k$-plane section of $X \subset \P^n$ if $X'$ is cut out by the restriction of $f$ to $\P^k$, where $X = V(f)$.

We now show that $F_1(X)$ contains no $k$-gonal curves for certain ranges of $n$, $d$ and $k$.

\begin{theorem}
\label{thm-noKgonalCurves}
If $n \leq \frac{d(d+3)}{6}+1-\frac{k}{3}$, then a very general degree $d$ hypersurface $X$ in $\P^n$ contains no $k$-gonal curves in $F_1(X)$.
\end{theorem}
\begin{proof}
Let $\uU \lL_{n,d}$ be the set of pairs $(\ell, X)$ of lines $\ell$ lying in a hypersurface $X$ of degree $d$ in $\P^n$. Let $R_{n,d,k}$ be the space of $(\ell, X)$ such that $F_1(X)$ has a $k$-gonal curve passing through $[\ell]$. We see that $R_{n,d,k}$ will be a countable union of varieties. If for some tuple $(n,d,k)$, the codimension of (each component of) $R_{n,d,k}$ in $\uU \lL_{n,d}$ is at least $2n-d-3$, then a general hypersurface $X$ of degree $d$ in $\P^n$ will contain no $k$-gonal curves in $F_1(X)$.

Let $d \geq 3$ be an integer, let $m = \frac{d(d+1)}{2}-k$, and let $X \subset \P^m$ be a general hypersurface of degree $d$ in $\P^m$. By Corollary \ref{cor-BVAFano}, we see that $R_{m,d,k} \subset \uU \lL_{m,d}$ has codimension at least $1$, so we can find some pair $(\ell_0, X_0)$ where $F_1(X_0)$ has no $k$-gonal curves through $[\ell_0]$.

Let $(\ell_1,X_1)$ be a general point of a component of $R_{n,d,k}$. We find a subvariety $S \subset \uU \lL_{n,d}$ containing $(\ell_1, X_1)$ such that $S \cap R_{n,d,k}$ is of codimension at least $2n-d-3$ in $S$. Since $(\ell_1, X_1)$ could have been on any component, it will follow that $R_{n,d,k}$ will have codimension at least $2n-d-3$ in $\uU \lL_{n,d}$. We now construct $S$.

We claim that we can find a hypersurface $Y \subset \P^M$ containing a line $\ell_2$ such that $(\ell_0, X_0)$ and $(\ell_1,X_1)$ are both parameterized linear sections of $(Y,\ell_2)$. To see this, choose coordinates $x_0, x_1, \dots, x_n$ on $\P^n$ so that $\ell_1$ is given by the vanishing of $x_2, \dots, x_n$, and choose coordinates $x_0, x_1, y_2, \dots, y_m$ on $\P^m$ so that $\ell_0$ is given by the vanishing of $y_2, \dots, y_m$. Write $X_0 = V(f_0)$ and $X_1 = V(f_1)$. Then we may take $Y = V(f_0+f_1)$ in $\P^{m+n-1}$, which will have the desired properties.

Let $S_n$ be the closure of the set of parameterized $n$-plane sections of $(Y, \ell_2)$. We see by the fact that $(\ell_0,X_0)$ has no $k$-gonal curves in $F_1(X_0)$ passing through $[\ell_0]$ that $S_m \cap R_{m,d,k}$ is codimension at least one in $S_m$. By Corollary \ref{cor-cutByHyperplane}, we see that $S_{m-1} \cap R_{m-1,d,k}$ has codimension at least $2$ in $S_{m-1}$. Thus, $S_n \cap R_{n,d,k}$ has codimension at least $m-n+1$ in $S_n$. Therefore, if $m-n+1 \geq 2n-d-3$, we see that there will be no $k$-gonal curves in $F_1(X)$ for $X$ a general hypersurface of degree $d$ in $\P^n$. This will hold if
\[ n \leq \frac{m+d+4}{3} = \frac{d(d+3)}{6}+\frac{4}{3} - \frac{k}{3}. \]
\end{proof}

\begin{proposition}
\label{prop-fewLinesThroughPoint}
If $X \subset \P^n$ is a general degree $d \geq n$ hypersurface, then any $p \in X$ has at most $n-1$ lines in $X$ passing through $p$.
\end{proposition}
\begin{proof}
It suffices to prove the result for $d = n$. Let $\uU$ be the incidence correspondence of pairs $(p,X)$ where $X \subset \P^n$ is a degree $n$ hypersurface and $p \in X$ is a point. Let $\lL_m \subset \uU$ be the set of pairs $(p,X)$ such that $X$ contains $m$ lines passing through $p$. We wish to show that $\lL_{m}$ has codimension at least $n$ in $\uU$ for $m \geq n$, from which the result will follow. Consider a point $p \in \P^n$. Choose coordinates on $\P^n$ so that $p$ is the point $[1,0,\dots,0]$. Then given an $X = V(f)$ containing $p$, we can expand the equation of $f$ around $p$, writing $f = f_1 x_0^{d-1} + f_2 x_0^{d-2} + \dots + f_d$. Then the space of lines in $X$ passing through $p$ will be $\Sigma_p = V(f_1, \dots, f_d) \subset \P^{n-1}$, where $\Sigma_p$ is naturally contained in the $\P^{n-1}$ of lines in $\P^n$ passing through $p$. We consider the codimension of the locus of $f$ for which $V(f_1, \dots, f_i)$ has larger than expected dimension. Many of the ideas here are adapted from \cite{HRS} and \cite{riedl-yang1}.

If $i \leq n-1$, the locus of $f$ where $V(f_i)$ contains a component of $V(f_1, \dots, f_{i-1})$ has codimension at least $\binom{n-i+i}{i} = \binom{n}{i} \geq n$, so the locus $S \subset \uU$ of the set of pairs $(p,X)$ where there is a positive dimensional family of lines through $p$ in $X$ has codimension at least $n$. Thus, we may assume $V(f_1, \dots, f_{n-1})$ is a finite set of points. For $f_n$ to contain $k$ of those points is $k$ conditions by Lemma \ref{lem-pointsConditions}, so the locus $\lL_m \subset \uU$ is codimension at least $m$. Thus, a general hypersurface $X$ of degree $d$ in $\P^n$ cannot have $n$ lines passing through a single point of $X$. The result follows.
\end{proof}

\begin{proof}[Proof of Corollary \ref{thm-kgonalCurvesInY}]
We simply put together the pieces. Let $\uU \to F_1(X)$ be the universal line on $X$, mapping to $F_1(X)$ via a map $\pi_1$. Then $\uU$ maps surjectively to $Y$ via a map $\pi_2$. Note that $\pi_2$ is finite by Proposition \ref{prop-fewLinesThroughPoint}. Suppose $C \subset Y$ is a $k$-gonal curve, and let $D$ be an irreducible component of $\pi_2^{-1}(C)$. Then by Proposition \ref{prop-fewLinesThroughPoint}, the degree of $\pi_2|_D$ is at most $n-1$, and so $D$ will have gonality at most $k(n-1)$. If $D$ is contracted by $\pi_1$, then $C$ must have been a line in $X$. If $D$ is not contracted by $\pi_1$, then its image will be a curve in $F_1(X)$ of gonality at most $k(n-1)$. By Theorem \ref{thm-noKgonalCurves}, this means $k(n-1) > \frac{d^2+3d+6}{2}-3n = \frac{d^2+3d+6-6n}{2}$. Thus, $k > \frac{d^2+3d+6-6n}{2(n-1)}$, so $Y$ contains no $k$-gonal curves other than lines for $k \leq \frac{d^2+3d+6-6n}{2(n-1)}$. In particular, this bound for $k$ is at least one for every $d \geq n \geq 3$.
\end{proof}


\newcommand{\closer}{\vspace{-1.5ex}}

\end{document}